\documentclass[11pt]{article}
\usepackage{amssymb,amsmath,amsthm,pslatex,times,url}
\usepackage{mathrsfs}

\title{On the metric dimension of Grassmann graphs} 
\author{Robert F.~Bailey and
Karen Meagher\footnote{Department of Mathematics and Statistics, University of Regina, 3737 Wascana Parkway, Regina, SK, S4S 0A2, Canada.  E-mail: \texttt{robert.bailey@uregina.ca}, \texttt{karen.meagher@uregina.ca}}
}

\newtheorem{thm}{Theorem}
\newtheorem{lemma}[thm]{Lemma}
\newtheorem{prop}[thm]{Proposition}

\theoremstyle{definition}
\newtheorem{defn}[thm]{Definition}

\newcommand{\rank}{\mathrm{rank}}

\begin{document}
\maketitle
\begin{abstract}
The {\em metric dimension} of a graph $\Gamma$ is the least number of vertices in a set with the property that the list of distances from any vertex to those in the set uniquely identifies that vertex.  We consider the Grassmann graph $G_q(n,k)$ (whose vertices are the $k$-subspaces of $\mathbb{F}_q^n$, and are adjacent if they intersect in a $(k-1)$-subspace) for $k\geq 2$. We find an upper bound on its metric dimension, which is equal to the number of $1$-dimensional subspaces of $\mathbb{F}_q^n$.  We also give a construction of a resolving set of this size in the case where $k+1$ divides $n$, and a related construction in other cases.
\end{abstract}

\section{Introduction}
In this paper, we are concerned with finding an upper bound on the metric dimension of Grassmann graphs.  We use the notation of \cite{bsmd}.  To define the metric dimension of a graph, we need the following concept.

\begin{defn} \label{defn:resolving1}
A {\em resolving set} for a graph $\Gamma=(V,E)$ is a set of vertices $S=\{v_1,\ldots,v_k\}$ such that for all $w \in V$, the list of distances $\mathcal{D}(w|S)=(d(w,v_1),\ldots,d(w,v_k))$ uniquely determines $w$.
\end{defn}

That is, $S$ is a resolving set for $\Gamma$ if for any pair of vertices $u,w$, $\mathcal{D}(u|S) = \mathcal{D}(w|S)$ if and only if $u=w$.  We call the list of distances $\mathcal{D}(w|S)$ the {\em code} of the vertex $w$, so $S$ is a resolving set if and only if each vertex has a different code.

\begin{defn} \label{defn:metdim}
The {\em metric dimension} of $\Gamma$, denoted $\mu(\Gamma)$, is the smallest size of a resolving set for $\Gamma$.
\end{defn}

Metric dimension was first introduced in the 1970s, independently by Harary and Melter \cite{HararyMelter76} and by Slater \cite{Slater75}.  In recent years, a considerable literature has developed (see \cite{bsmd} for details).  A particularly interesting case is that of {\em distance-regular graphs} (we refer the reader to the book by Brouwer, Cohen and Neumaier \cite{BCN} for background on this topic).  In fact, in the case of primitive distance-regular graphs, bounds on a parameter equivalent to the metric dimension were obtained in 1981 by Babai \cite{Babai81} (where they were used to obtain combinatorial bounds on the possible orders of primitive permutation groups: see \cite{bsmd} for further details).

Many families of distance-regular graphs are so-called ``graphs with classical parameters'' (see \cite[Chapter 9]{BCN}): these include the well-known {\em Hamming graphs} and {\em Johnson graphs}, for which metric dimension has already been studied.  In the case of Hamming graphs, various results on the metric dimension have been obtained (although not always phrased in these terms: see \cite{bsmd}). For the Johnson graphs $J(n,k)$, where the vertices are $k$-subsets of an $n$-set, and two $k$-subsets are adjacent if they intersect in a $(k-1)$-set, the exact value is known for $k=2$: this was obtained by Cameron and the first author in \cite{bsmd}.  For $k \geq 3$ bounds have been obtained by the authors and others in \cite{johnson,Caceres09}; the following result appears in \cite{johnson}.

\begin{thm} \label{thm:Johnson}
The metric dimension of the Johnson graph $J(n,k)$ is at most $n$. 
\end{thm}

After the Johnson graphs, the obvious next family to consider is the Grassmann graphs.  Throughout this paper, $V(n,q)$ denotes the $n$-dimensional vector space over the finite field $\mathbb{F}_q$.

\begin{defn} \label{defn:Grassmann}
The {\em Grassmann graph} $G_q(n,k)$ has as its vertex set the set of all $k$-dimensional subspaces of $V(n,q)$, and two vertices are adjacent if the corresponding subspaces intersect in a subspace of dimension $k-1$.
\end{defn}

Note that if $k=1$, we have a complete graph, so we shall assume that $k\geq 2$.  Also, it is not difficult to show that $G_q(n,k) \cong G_q(n,n-k)$, so it suffices to consider $k \leq n/2$.  The number of vertices is simply the Gaussian binomial coefficient:
\[ {n \brack k}_q = \prod_{i=0}^{k-1} \frac{(q^n-q^i) }{(q^k-q^i) }. \]
The Grassmann graph is a distance-regular graph of diameter $k$: the distance between two vertices $X$ and $Y$ is $k-\dim(X\cap Y)$.  It is considered to be the ``$q$-analogue'' of the Johnson graph $J(n,k)$.  

Having defined the Grassmann graphs, we are now ready to state our main result.

\begin{thm} \label{thm:main}
Let $G_q(n,k)$ be the Grassmann graph, where $2\leq k \leq n/2$.  Then the metric dimension of $G_q(n,k)$ is at most ${n \brack 1}_q$.
\end{thm}

This bound, and its proof, are inspired by a similar result for the Johnson graph $J(n,k)$ given in Theorem~\ref{thm:Johnson} above.

\section{Proof of the bound}
The proof of Theorem~\ref{thm:main} uses linear algebra, and is similar to an analogous proof given in \cite{johnson} for the Johnson graphs.  We observe that a collection of $k$-subspaces $\{U_1,\ldots,U_m\}$ of $V(n,q)$ is a resolving set for the Grassmann graph $G_q(n,k)$ if and only if, for any pair of distinct $k$-subspaces $A,B \leq V(n,q)$, we have
\begin{equation}\label{eqn:dim}
(\dim(A\cap U_1),\ldots,\dim(A\cap U_m)) \neq (\dim(B\cap U_1),\ldots,\dim(B\cap U_m)). 
\end{equation}
Consider the collection of all 1-dimensional subspaces of $V(n,q)$, and label these as $L_1,\ldots,L_N$ (where $N={n \brack 1}_q$).  The {\em incidence vector} of a $k$-subspace $U$ is the vector $(u_1,\ldots,u_N)\in\mathbb{R}^N$ whose entries are
\[ u_i = \left\{ \begin{array}{ccl} 1 & \,\, & \textnormal{if $L_i \leq U$} \\ 0 && \textnormal{otherwise.} \end{array} \right. \]
Furthermore, given a collection of $k$-subspaces $\mathcal{S} = \{U_1,\ldots,U_m\}$, the {\em incidence matrix} of $\mathcal{S}$ is the $m\times N$ matrix whose rows are the incidence vectors of $U_1,\ldots,U_m$.

\begin{lemma}\label{lemma:rank}
If $M$ is the incidence matrix of a family $\mathcal{S} = \{U_1,\ldots,U_m\}$ of $k$-subspaces of $V(n,q)$, and $\rank(M)={n \brack 1}_q$, then $\mathcal{S}$ is a resolving set for $G_q(n,k)$.
\end{lemma}

\proof Suppose we have two distinct $k$-subspaces $A,B\leq V(n,q)$, whose incidence vectors are $a,b$ respectively.  Now, consider the entries of the vector $Ma$: by construction, we have that the $i^{\textnormal{th}}$ entry of $Ma$ is precisely the number of 1-dimensional subspaces in $A\cap U_i$, which is ${d \brack 1}_q$ (where $d= \dim(A\cap U_i)$).  Consequently, it follows that if $Ma\neq Mb$, then 
condition (\ref{eqn:dim}) holds.  Also, it is clear that if $\rank(M)={n \brack 1}_q$, the linear transformation represented by $M$ is one-to-one, and thus $Ma\neq Mb$ for any distinct vectors $a,b$.

Hence, if $\rank(M)={n \brack 1}_q$, condition (\ref{eqn:dim}) holds for any distinct $k$-subspaces $A,B$, and therefore $\mathcal{S}$ is a resolving set for $G_q(n,k)$.
\endproof

So to prove Theorem~\ref{thm:main}, it remains to show that a family of subsets which such an incidence matrix actually exists.  We will do this by showing that the incidence matrix of {\em all} $k$-subspaces has rank ${n \brack 1}_q$.  Recall that, for any real matrix~$M$, we have $\rank(M^{\mathrm{T}} M) = \rank(M)$.

\begin{lemma}\label{lemma:rank2}
Let $M$ be the incidence matrix of the set of all $k$-subspaces of $V(n,q)$.  Then $\rank(M)={n \brack 1}_q$.
\end{lemma}

\proof Consider the matrix $M^{\mathrm{T}} M$, which is an ${n \brack 1}_q \times {n \brack 1}_q$ square matrix.  Now, the diagonal entries of $M^{\mathrm{T}} M$ are all ${n-1 \brack k-1}_q$: this corresponds to the number of $k$-subspaces containing a given $1$-subspace.  Also, the off-diagonal entries of $M^{\mathrm{T}} M$ are all ${n-2 \brack k-2}_q$: this corresponds to the number of $k$-subspaces containing a given pair of $1$-subspaces.  Therefore, we have
\[ M^{\mathrm{T}} M = {n-2 \brack k-2}_q J + \left( {n-1 \brack k-1}_q - {n-2 \brack k-2}_q \right)I \]
(where $I$ is the identity matrix and $J$ the all-ones matrix).  It is then a straightforward exercise to see that $\det(M^{\mathrm{T}} M)\neq 0$, and thus $\rank(M)=\rank(M^{\mathrm{T}} M)={n \brack 1}_q$.  \endproof

\begin{proof}[Proof of Theorem~\ref{thm:main}]
Let $M$ denote the incidence matrix of all $k$-subspaces of $V(n,q)$.  By Lemma~\ref{lemma:rank2}, this matrix has rank ${n \brack 1}_q$; in particular, $M$ has an invertible ${n \brack 1}_q\times {n \brack 1}_q$ submatrix $M'$.  By Lemma~\ref{lemma:rank}, $M'$ is the incidence matrix of a resolving set for $G_q(n,k)$ of size ${n \brack 1}_q$.  Consequently the metric dimension of $G_q(n,k)$ is at most ${n \brack 1}_q$.
\end{proof}

\section{Constructions of resolving sets}
The disadvantage of the proof of Theorem~\ref{thm:main} given above is that it does not provide an explicit construction of a resolving set for the Grassmann graph $G_q(n,k)$.  In this section, we address this issue; once again, we adapt earlier results for Johnson graphs.

In~\cite{Caceres09}, C\'aceres {\em et al.}\ give a construction of a resolving set for the Johnson graph $J(n,k)$, which has size $(k+1) \lceil n/(k+1) \rceil$.  This is explained most easily in the case where $k+1$ divides $n$, where one partitions the $n$-set into a collection of disjoint $(k+1)$-sets, and in each part taking all $k$-subsets, to obtain a resolving set for $J(n,k)$.  Various refinements can be made to obtain slight improvements to the size, but the principle of the construction remains the same.  (See \cite{johnson} for details.)

To obtain resolving sets for the Grassmann graphs, we will give separate constructions for two cases: the first, when $n$ is divisible by $k+1$, is considered in Proposition \ref{prop:divides}; the second, when $n$ is not divisible by $k+1$, is considered in Proposition \ref{prop:otherwise}.  We could have combined the two separate cases into one construction, but in doing so risks losing the intuition of how the constructions arise.

Our constructions require some notions from finite geometry.

\begin{defn} \label{defn:spread}
A $t$-{\em spread} of $V(n,q)$ is collection of $t$-subspaces $\{W_1,\ldots,W_m\}$ with the following properties:
\begin{itemize}
\item[(i)] any non-zero vector $x\in V(n,q)$ belongs to exactly one $W_i$;
\item[(ii)] if $i\neq j$, then $W_i\cap W_j = \{0\}$.
\end{itemize}
\end{defn}

A classical result (see Dembowski \cite[p.\ 29]{Dembowski68}) shows that a $t$-spread of $V(n,q)$ exists if and only if $t$ is a divisor of $n$.  The number of subspaces in a $t$-spread is necessarily $(q^n-1)/(q^t-1)$.

\begin{prop} \label{prop:divides}
Suppose $k+1$ divides $n$, and that $\mathcal{S}=\{W_1,\ldots,W_m\}$ is a $(k+1)$-spread of $V(n,q)$.  Let $\mathcal{M}$ denote the union of the collections of all $k$-subspaces of each $W_i$; that is,
\[ \mathcal{M} = \bigcup_{i=1}^m \{U\leq W_i \, : \, \dim(U)=k \}. \]
Then $\mathcal{M}$ is a resolving set for $G_q(n,k)$ of size ${n \brack 1}_q$.
\end{prop}

\proof 
We will show that $\mathcal{M}$ is a resolving set by demonstrating that whenever $A$ and $B$ are two distinct $k$-subspaces of $V(n,q)$, there exists some $U\in \mathcal{M}$ with $\dim(A\cap U)\neq \dim(B\cap U)$.

We consider how $A$ and $B$ intersect with the members $W_1,\ldots,W_m$ of the $(k+1)$-spread $\mathcal{S}$.

\begin{description}
\item[Case 1.] If either of $A$ or $B$ (say $A$) is contained within some member of $\mathcal{S}$, then clearly $A\in\mathcal{M}$, and we are done (as $\dim(A\cap B) \neq k$).  
\end{description}

So we suppose not.  For each $i$, let $A_i=A\cap W_i$ and $B_i=B\cap W_i$.  There are now two possible cases.

\begin{description}
\item[Case 2.] Suppose there exists some $i$ where $\dim(B_i)<\dim(A_i)$.  Since $\dim(A_i)<k$, there must exist some $k$-subspace $U$ of $W_i$ which contains $A_i$ as a subspace.  Therefore $\dim(A_i\cap U)=\dim(A_i)$, while $\dim(B_i\cap U) \leq \dim(B_i) < \dim(A_i)$.  In particular, we have $\dim(A_i\cap U) \neq \dim(B_i\cap U)$, and thus $\dim(A\cap U) \neq \dim(B\cap U)$, where $U\in\mathcal{M}$.

\item[Case 3.] The remaining scenario is that $\dim(A_i)=\dim(B_i)$ for all $i$.  However, there must exist some $i$ where $A_i \neq B_i$ (and where both are non-zero); otherwise we would have $A=B$.  For that particular value of $i$, suppose $\dim(A_i)=\dim(B_i)=d<k$.

Now, for any $k$-subspace $U$ of $W_i$, we must have $\dim(A_i\cap U)\geq d-1$ and $\dim(B_i\cap U)\geq d-1$ (this follows from the identity $\dim(A_i\cap U)=\dim(A_i)+\dim(U)-\dim(A_i+U)$, and that $\dim(A_i+U)\leq\dim(W_i)=k+1$; likewise for $\dim(B_i\cap U)$).  Thus, for there to be a $k$-subspace $U$ with $\dim(A_i\cap U) \neq \dim(B_i\cap U)$, we must have (without loss of generality) that $A_i$ is a subspace of $U$, and $B_i$ is not.  We aim to construct such a $U$.

Take a basis $\{a_1,\ldots,a_d\}$ for $A_i$.  Since $A_i\neq B_i$, there exists a vector $x\in B_i\setminus A_i$.  Now, because $x\not\in A_i$, we have that the set $\{a_1,\ldots,a_d,x\}$ is linearly independent.  Extend this to a basis $\{ a_1,\ldots,a_d,x,w_1,\ldots,w_{k-d}\}$ for $W_i$.  Now let $U$ be the $k$-subspace spanned by the vectors $\{a_1,\ldots,a_d,w_1,\ldots,w_{k-d}\}$: by construction, $A_i$ is a subspace of $U$, but $B_i$ is not, since $x\in B_i\setminus U$.  In particular, we see that $d = \dim(A_i\cap U) \neq \dim(B_i\cap U) =d-1$, and thus $\dim(A\cap U) \neq \dim(B\cap U)$, where $U\in\mathcal{M}$.
\end{description}

In all cases, we have $U\in \mathcal{M}$ with $\dim(A\cap U)\neq \dim(B\cap U)$, so therefore $\mathcal{M}$ is a resolving set.  Finally, we observe that 
\[ |\mathcal{M}| = \frac{q^n-1}{q^{k+1}-1} {k+1 \brack 1}_q = \frac{q^n-1}{q^{k+1}-1} \frac{q^{k+1}-1}{q-1} = {n \brack 1}_q. \]
\endproof

We remark that Proposition~\ref{prop:divides} provides an alternative proof of Theorem~\ref{thm:main} in the case where $k+1$ divides $n$.  We also remark that Lemma~\ref{lemma:rank} can be applied to show that the set $\mathcal{M}$ of Proposition~\ref{prop:divides} is a resolving set for $G_q(n,k)$: the incidence matrix of $\mathcal{M}$ is a block-diagonal matrix, where each block is the incidence matrix of the collection of all $k$-subspaces of $V(k+1,q)$.  By Lemma~\ref{lemma:rank2}, this has full rank, and thus so does the incidence matrix of $\mathcal{M}$.

We would also like to obtain a construction of a resolving set when $k+1$ does not divide $n$.  In that situation, there is no $(k+1)$-spread of $V(n,q)$; however, a result of Beutelspacher \cite{Beutelspacher78} provides an alternative.  Following his notation, where $T$ is a set of positive integers, a {\em $T$-partition} of $V(n,q)$ is a partition of the non-zero vectors of $V(n,q)$ into subpaces whose dimensions form the set $T$.  Thus if $T=\{t\}$, then a $T$-partition of $V(n,q)$ is simply a $t$-spread of $V(n,q)$.  The following lemma is an immediate consequence of Lemma 3 in Beutelspacher's 1978 paper \cite{Beutelspacher78}.

\begin{lemma}[Beutelspacher \cite{Beutelspacher78}] \label{lemma:partition}
Suppose $n=r(k+1)+t$, where $0<t<k+1$.  Then there exists a $\{k+1,t\}$-partition of $V(n,q)$.
\end{lemma}

Beutelspacher's construction works as follows: write $s=r(k+1)$ (so that $n=s+t$) and take a $(k+1)$-spread of $V(s,q)$.  The remaining $q^n-q^s$ vectors in $V(n,q)\setminus V(s,q)$ are then partitioned into the non-zero vectors of $t$-subspaces (of which there are necessarily $q^s$).  In what follows, we suppose that $\mathcal{P}$ is a $\{k+1,t\}$-partition of $V(n,q)$.  We write this as $\mathcal{P}=\mathcal{S}\cup\mathcal{T}$, where $\mathcal{S}=\{W_1,\ldots,W_m\}$ is a $(k+1)$-spread of $V(s,q)$, and where $\mathcal{T}=\{X_1,\ldots,X_\ell\}$ consists of $t$-subspaces covering the remaining vectors.  We also let $Z$ denote a fixed \mbox{$(k-t+1)$}-dimensional subspace of $V(s,q)$.

\begin{prop} \label{prop:otherwise}
Suppose $n=s+t$, where $(k+1)$ divides $s$ and $0<t<k+1$.  Let $\mathcal{P}=\mathcal{S}\cup\mathcal{T}$ and $Z$ be as above.  For each $W_i \in \mathcal{S}$, take all the $k$-subspaces (as in Proposition \ref{prop:divides}); for each $X_j\in \mathcal{T}$, extend to the $(k+1)$-subspace $X_j \oplus Z$, and take all $k$-subspaces of this.  Let $\mathcal{M}$ denote the union of these collections: that is,
\[ \mathcal{M} = \bigcup_{i=1}^m \{U\leq W_i \, : \, \dim(U)=k \} \cup \bigcup_{j=1}^\ell \{ U \leq X_j \oplus Z \, : \, \dim(U)=k\}. \]
Then $\mathcal{M}$ is a resolving set for $G_q(n,k)$, of size at most ${(n-t)+(k+1) \brack 1}_q$.
\end{prop}

\proof To show that $\mathcal{M}$ is a resolving set, we show that whenever $A$ and $B$ are two distinct $k$-subspaces of $V(n,q)$, there exists some $U\in \mathcal{M}$ with $\dim(A\cap U)\neq \dim(B\cap U)$.

If $A$ or $B$ are entirely contained within $V(s,q)$, then we are done; by the arguments used in the proof of Proposition \ref{prop:divides}, $\mathcal{M}$ contains a resolving set for the subgraph $G_q(s,k)$. Also, if $A\cap V(s,q)\neq B\cap V(s,q)$, the proof of Proposition \ref{prop:divides} shows that a vertex of the resolving set for $G_q(s,k)$ is able to distinguish $A$ and $B$.  Thus the only situation remaining is when both $A$ and $B$ intersect non-trivially with $V(n,q)\setminus V(s,q)$, and where $A\cap V(s,q) = B\cap V(s,q)$.

Since $A\neq B$ but $A\cap V(s,q) = B\cap V(s,q)$, it follows that there exists a $t$-subspace $X_i\in \mathcal{T}$ where $A\cap X_i \neq B\cap X_i$.  Suppose without loss of generality that $\dim(A\cap X_i) \leq \dim(B\cap X_i)$.  Then there exists a vector $x \in (B\cap X_i)\setminus A$.  It follows that $x\in B\cap (X_i \oplus Z)$ but $x\not\in A\cap (X_i \oplus Z)$; consequently, we have $A\cap (X_i \oplus Z) \neq B\cap (X_i \oplus Z)$.  By the arguments in Cases 2 and 3 of Proposition~\ref{prop:divides}, it follows that there exists a $k$-subspace $U\leq X_i\oplus Z$ satisfying $\dim(A\cap U) \neq \dim(B\cap U)$.  By construction, $U\in \mathcal{M}$. 

Finally, we obtain the bound by observing that 
\begin{eqnarray*} 
|\mathcal{M}| & \leq & {s \brack 1}_q + q^s{k+1 \brack 1}_q \\
              & = & \frac{q^s-1}{q-1} + q^s \frac{q^{k+1}-1}{q-1} \\
              & = & \frac{q^{s+k+1}-1}{q-1} \\
              & = & {(n-t)+(k+1) \brack 1}_q.
\end{eqnarray*}
\endproof

We remark that the bound on the size of $\mathcal{M}$ given in Proposition~\ref{prop:otherwise} is likely to be an over-estimate, for two reasons.  First, it is possible that for distinct $X_i,X_j$, we may have $X_i \oplus Z = X_j \oplus Z$; second, even if $X_i \oplus Z \neq X_j \oplus Z$, they may have a common $k$-subspace.  The precise number of repetitions will be dependent on the choice of $Z$ and the structure of the partition $\mathcal{P}=\mathcal{S}\cup\mathcal{T}$, and thus counting the actual size of $\mathcal{M}$ will be difficult in general.  However, in the case $t=1$, it is straightforward.  First, if $t=\dim(X_i)=1$, we can count precisely the number of $1$-subspaces in $(X_i\oplus Z)\setminus Z$ to be $q^k$; second, if $(X_i \oplus Z) \neq (X_j \oplus Z)$, the only $k$-subspace that can be contained in $(X_i \oplus Z) \cap (X_j \oplus Z)$ is $Z$ itself. 
Since the choice of $Z$ was arbitrary, we can take $Z$ to be one of the $k$-subspaces of some $W_j$.  Hence, if $t=1$, we have
\[ |\mathcal{M}| = {n-1 \brack 1}_q + \frac{q^{n-1}}{q^k} \left( {k+1 \brack 1}_q - 1\right) = {n \brack 1}_q + q^{n-k} {k-1 \brack 1}_q. \]

\section{Discussion}
A natural question is to compare our result in Theorem~\ref{thm:main} with the previously-known upper and lower bounds on metric dimension.

An implicit, approximate lower bound can be obtained using \cite[Proposition 3.6]{bsmd}, which in the case of Grassmann graphs yields
\[ \mu(G_q(n,k)) \gtrapprox \log_k {n \brack k}_q. \]
This bound is clearly much smaller than the upper bound we obtained in Theorem~\ref{thm:main}.  However, this lower bound is obtained only by considering the diameter of a graph and not its structure, and is often far from the actual value of $\mu$.  (See \cite[\S3.1]{bsmd} for a more detailed discussion.)

As mentioned in the introduction, there are known upper bounds due to Babai \cite{Babai81} which can be applied here (see \cite{bsmd} for an interpretation of these in terms of metric dimension of distance-regular graphs).  For the Grassmann graphs, Babai's most general bound (see \cite[Theorem 2.1]{Babai81}; see also \cite[Theorem 3.15]{bsmd}) yields
\begin{equation}
 \mu(G_q(n,k)) < 4\sqrt{ {n \brack k}_q } \log {n \brack k}_q 
\end{equation}
while his stronger bound (see \cite[Theorem 2.4]{Babai81}; see also \cite[Theorem 3.22]{bsmd}) yields
\begin{equation}
 \mu(G_q(n,k)) < 2k  \frac{ {n \brack k}_q }{ {n \brack k}_q - M } \log {n \brack k}_q 
\end{equation}
where 
\[ M = \max_{0\leq j \leq k} q^{j^2} {n-k \brack j}_q {k \brack j}_q. \]
These bounds are difficult to evaluate exactly, so we conducted some experiments using \textsc{Maple} to compare these bounds with the one obtained in Theorem \ref{thm:main}.  Our experiments indicate that for $k>2$, our constructive bound of ${n \brack 1}_q$ is an improvement on Babai's weaker bound.  When $k=2$ and $q$ is large, Babai gives a better bound.  They also suggest that Babai's stronger bound is, for fixed $q$ and $n$, is descreasing in $k$ (in comparison, our bound is independent of $k$ and thus stays fixed), and gives a better bound for larger values of $k$. 
However, it should be mentioned that Babai's results are obtained using a result of Lov\'asz on fractional covers in hypergraphs \cite{Lovasz75}, and are not explicit, whereas (in the case where $k+1$ divides $n$), our bound is met by an explicit construction of a resolving set.  Also, the proof of Theorem~\ref{thm:main} implictly gives a method of constructing resolving sets, by sequentially adding $k$-subspaces so that the incidence matrix has full rank. 

Another natual question concerns the Grassmann graphs being the $q$-analogue of the Johnson graphs, and comparing our result with known bounds for the metric dimension of those graphs.  Often, when one has obtained an invariant of the Grassmann graph $G_q(n,k)$ as a function of $q$, then by taking the limit as $q\to 1$, one obtains the same invariant for the Johnson graph $J(n,k)$.  Now, if we consider our bound from Theorem~\ref{thm:main}, it is easy to see that
\[ \lim_{q\to 1} {n \brack 1}_q = n, \]
which is precisely the bound for the Johnson graphs in Theorem~\ref{thm:Johnson}. 

Finally, we remark that it is likely that one can refine the constructions in Propositions \ref{prop:divides} and \ref{prop:otherwise} (in the manner of the results for Johnson graphs in \cite{johnson}) to obtain a tighter bound on the metric dimension.  However, unlike the case of the Johnson graphs, we believe it is unlikely that such refinements will affect the order of magnitude of the bound too much, or that they would have such a tidy form.

\subsection*{Acknowledgements}
\label{sec:ack}
The authors would like to thank J.\ C\'aceres for first communicating the results of \cite{Caceres09}, and P.~J.~Cameron and K.~Wang for pointing out errors in earlier versions of this paper.  R.~F.~Bailey acknowledges support from a PIMS Postdoctoral Fellowship.  K.~Meagher acknowledges support from an NSERC Discovery Grant.

\end{document}